\documentclass[11pt,reqno,psamsfonts]{amsart}
\pdfoutput=1

\usepackage{amssymb}
\usepackage{amsthm}
\usepackage{amsmath}
\usepackage{latexsym}
\usepackage[T1]{fontenc}
\usepackage[utf8]{inputenc}
\usepackage[russian, french, 
english]{babel}

\usepackage{graphicx}
\usepackage{wrapfig}
\usepackage{mathtools}
\usepackage{amsbsy}
\usepackage[inline]{enumitem}
\usepackage{mathrsfs}
\usepackage{array}
\usepackage{multicol}
\usepackage{stmaryrd}
\usepackage{cancel}
\usepackage{lmodern}
\usepackage{mathabx}
\usepackage{upgreek}
\usepackage{titlesec}
\usepackage{titletoc}
\usepackage[spacing=true,kerning=true,babel=true,tracking=true]{microtype}
\usepackage[shortcuts]{extdash}
\usepackage[foot]{amsaddr}
\usepackage[left=1in,right=1in,top=1in,bottom=1in,bindingoffset=0cm]{geometry}
\usepackage{bm}
\usepackage{centernot}
\usepackage{mdframed}
\usepackage[hidelinks]{hyperref}
\usepackage{xspace}
\usepackage[most]{tcolorbox}
\usepackage{framed}
\usepackage[labelfont=bf]{caption}
\usepackage{tikz}
\usetikzlibrary{math}
\usetikzlibrary{shapes,snakes}
\usetikzlibrary{arrows.meta}
\usetikzlibrary{decorations.pathmorphing}
\usetikzlibrary{patterns}
\usetikzlibrary{positioning,calc,intersections}
\usepackage{tikzsymbols}
\usepackage{float}

\DeclareFontFamily{U}{skulls}{}
\DeclareFontShape{U}{skulls}{m}{n}{ <-> skull }{}

\usepackage[
    sortcites,
    backend=biber, style=alphabetic, sorting=nyt, maxnames=100,backref=true]{biblatex}
\newtheoremstyle{bfnote}%
{}{}%
{\slshape}{}%
{\bfseries}{\bfseries.}%
{ }%
{\thmname{#1}\thmnumber{ #2}\thmnote{ \ep{\normalfont{}#3}}}

\theoremstyle{bfnote}
\newtheorem{theo}{Theorem}[section]
\newtheorem{theorem}[theo]{Theorem}
\newtheorem*{theo*}{Theorem}
\newtheorem{proposition}[theo]{Proposition}
\newtheorem{lemma}[theo]{Lemma}

\newtheorem{corollary}[theo]{Corollary}

\newtheorem*{corl*}{Corollary}

\theoremstyle{definition}
\newtheorem{definition}[theo]{Definition}
\newtheorem*{defn*}{Definition}

\newtheorem{question}[theo]{Question}

\newtheorem*{remks*}{Remarks}
\newtheorem*{exmp*}{Example}

\theoremstyle{remark}
\newtheorem*{ques*}{Question}
\newtheorem*{remk*}{Remark}

\newcommand*{\myproofname}{Proof}

\makeatletter
\newcommand{\neutralize}[1]{\expandafter\let\csname c@#1\endcsname\count@}
\makeatother

\newcommand{\Q}{\mathbb{Q}}

\renewcommand{\epsilon}{\varepsilon}

\renewcommand{\phi}{\varphi}
\renewcommand{\theta}{\vartheta}
\renewcommand{\leq}{\leqslant}
\renewcommand{\geq}{\geqslant}

\newcommand{\bemph}[1]{{\normalfont#1}}
\newcommand{\ep}[1]{\bemph{(}#1\bemph{)}}

\newcommand{\emphd}[1]{{\fontseries{b}\selectfont\textsf{#1}}}

\newcommand{\N}{\mathbb{N}}
\newcommand{\Z}{\mathbb{Z}}

\newcommand{\asdim}{\mathrm{asdim}}

\newcommand{\asdimB}{\mathrm{asdim}_{\mathsf B}}
\newcommand{\rk}{\mathrm{rk}}
\newcommand{\gp}{\mathrm{gp}}
\newcommand{\free}{\mathcal F}
\newcommand{\period}{\mathcal P}
\newcommand{\univ}{\mathcal U}

\numberwithin{equation}{section}

\titleformat{\section}[block]{\large\bfseries\sffamily}{\thesection.}{1ex}{}
\titleformat{\subsection}[block]{\bfseries\sffamily}{\thesubsection.}{1ex}{}
\titleformat{\subsubsection}[runin]{\bfseries}{\bfseries\upshape\thesubsubsection.}{1ex}{}[.]

\titlespacing*{\section}{0pt}{*3}{*1}
\titlespacing*{\subsection}{0pt}{*3}{*1}
\titlespacing*{\subsubsection}{0pt}{*1.5}{*1}

 \titlecontents{section}
 [1.5em] %
 {\smallskip}
 {\bfseries\thecontentslabel\hspace{1.02em}}
{\bfseries}
 {\,\,\titlerule*[0.77pc]{}\bfseries\contentspage}
\titlecontents{subsection}
 [4em] %
 {\smallskip}
 {\thecontentslabel\hspace{1.02em}}
{\hspace*{2.32em}}
 {\,\,\titlerule*[0.77pc]{.}\contentspage}

\renewbibmacro{in:}{}

\renewbibmacro*{volume+number+eid}{%
	\printfield{volume}%
	\setunit*{\addnbspace}
	\printfield{number}%
	\setunit{\addcomma\space}%
	\printfield{eid}}

\DeclareFieldFormat[article]{volume}{\textbf{#1}\space}
\DeclareFieldFormat[article]{number}{\mkbibparens{#1}}

\DeclareFieldFormat{journaltitle}{#1,}
\DeclareFieldFormat[thesis]{title}{\mkbibemph{#1}\addperiod}
\DeclareFieldFormat[article, unpublished, thesis]{title}{\mkbibemph{#1},}
\DeclareFieldFormat[book]{title}{\mkbibemph{#1}\addperiod}
\DeclareFieldFormat[unpublished]{howpublished}{#1, }

\DeclareFieldFormat{pages}{#1}

\DeclareFieldFormat[article]{series}{Ser.~#1\addcomma}

\setlist{topsep=3pt,itemsep=3pt}

\title{\sffamily Asymptotic dimension of commutative monoid actions}
\date{}

\author{Jinmin Wang}
\address[JW]{\normalfont Institute of Mathematics, Chinese Academy of Sciences, Beijing, China}
\email{jinmin@amss.ac.cn}

\author{Jing Yu}
\address[JY]{\normalfont Shanghai Center for Mathematical Sciences, Fudan University, Shanghai, China}
\email{jyu@fudan.edu.cn}

\author{Jingming Zhu}
\address[JZ]{\normalfont College of Data Science, Jiaxing University, Jiaxing, China}
\email{jingmingzhu@zjxu.edu.cn}

\thanks{JW is partially supported by the National Natural Science Foundation of China (NSFC) 12501169. JY is partially supported by the NSFC 12371343 and 12525110 (PI: Hehui Wu).}

\begin{document}


\maketitle

\begin{abstract}
    We prove that every bounded-to-one action of a finitely generated commutative monoid has asymptotic dimension at most the torsion-free rank of its group completion, with control uniform for a fixed monoid, generating set, and generator fiber bound. The estimate is sharp. Together with the equality of classical and Borel asymptotic dimensions for these actions, it gives the sharp Borel rank bound and recovers hyperfiniteness for bounded-to-one Borel actions of countable commutative monoids. 
\end{abstract}

\section{Introduction}
Borel combinatorics studies combinatorial constructions on standard Borel spaces under the additional requirement that the objects constructed be Borel. A basic problem is to understand the countable Borel equivalence relations generated by Borel maps and actions. Of particular interest is hyperfiniteness: whether such a relation can be expressed as an increasing union of finite Borel equivalence relations.
A countable Borel equivalence relation $E$ is \emphd{hyperfinite} if 
\[E = \bigcup_{j \in \N} E_j, \qquad E_0 \subseteq E_1 \subseteq \cdots\]
where each $E_j$ is a Borel equivalence relation with finite classes. 

Two relevant cases are well understood. Gao and Jackson \cite{GJ15} proved that every Borel action of a countable abelian group has hyperfinite orbit equivalence relation. For a countable-to-one Borel map $f: X \to X$, where $X$ is a standard Borel space, Dougherty, Jackson, and Kechris \cite{DJK94} proved hyperfiniteness of the relation
\[x E_f y \iff f^{a}(x) = f^b(y) \textrm{ for some } a, b \in \N.\]
The first result concerns invertible maps, while the second allows noninvertibility but only one generating map. Commuting families of noninvertible maps bring these two settings together. Their composites form a commutative monoid, so monoid actions provide a natural language in which to retain the relations among the maps.

Shinko, Weilacher, and the second-named author \cite{SWY} proved that every bounded-to-one Borel action of a countable commutative monoid has hyperfinite orbit equivalence relation. Here a map is \emphd{bounded-to-one} if its fibers are uniformly finite. An action is \emphd{bounded-to-one} if every monoid element acts by such a map; the bound is allowed to depend on the monoid element. This paper concerns the sharp asymptotic dimension bound underlying these actions.

\subsection{Main results}

A \emphd{monoid} is a set with an associative multiplication and an identity element $1$. An action $M\curvearrowright X$ is a family of maps $x\mapsto ax$, $a\in M$, satisfying $1x=x$ and $(ab)x=a(bx)$. A set with such an action is an \emphd{$M$-set}. We consider commutative monoids. Given a finite generating set $S\subseteq M$, the \emphd{Schreier graph} $G_S^X$ has an undirected edge $\{x,sx\}$ for every $s\in S$ and $x\in X$, with loops discarded. Its path metric $\rho_X$ takes the value $\infty$ between distinct connected components. 

At scale $r$, an \emphd{$r$-component} of a subset is a maximal set connected by finite chains whose successive distances are at most $r$. The inequality $\asdim X\le n$ means that, at every scale $r\ge1$, the space can be covered by $n+1$ sets whose $r$-components have uniformly bounded diameter. The \emphd{control function} denotes the uniform diameter bound as a function of $r$. For a Schreier graph, this dimension does not depend on the finite generating set, and we can alternatively write it as $\asdim(M\curvearrowright X)$.

The \emphd{group completion} $M^{\mathrm{gp}}$ is obtained by formally adjoining inverses to the elements of $M$. For finitely generated commutative $M$, it is a finitely generated abelian group. We use the notation
\[
 M^{\mathrm{gp}}\cong\mathbb Z^n\oplus T,
 \qquad T\text{ finite},\qquad \operatorname{rk}(M)=n.
\]
The finite torsion subgroup does not contribute to the rank. Further algebraic details are given in Section~\ref{sec:pre}.

Shinko, Weilacher, and the second-named author \cite{SWY} raised the following question:
\begin{question}[{Shinko--Weilacher--Yu \cite[Question 5.17]{SWY}}]
    Let $M$ be a finitely generated commutative monoid
    and $X$ a bounded-to-one $M$-set.
    Is $\asdim(M \curvearrowright X) < \infty$?
    Is $\asdim(M \curvearrowright X) \le \rk(M)$? 
\end{question}

Our main theorem gives the sharp bound:   
\begin{theorem}\label{thm:rank}
Let $M$ be a finitely generated commutative monoid, let $S$ be a finite generating set, and let $K\geq1$ be an integer. For every bounded-to-one action $M\curvearrowright X$ satisfying
\[
 \sup_{x\in X,\ s\in S}|s^{-1}(x)|\leq K,
\]
we have
\[
\asdim(M \curvearrowright X) \le \rk(M).
\]
 The control function only depends on $M,S,K$. This dimension bound is sharp.
\end{theorem}

The translation action $\N^n\curvearrowright\Z^n$ gives the standard lattice, whose asymptotic dimension is $n$.
A commuting family of $k$ maps defines an $\N^k$-action, so we obtain the following consequence.

\begin{corollary}\label{cor:maps}
Let $k,K\geq1$ be integers, and let $f_1,\ldots,f_k\colon X\to X$ be pairwise commuting maps with fibers of cardinality at most $K$. Then the graph with edges $\{x,f_j(x)\}$ has asymptotic dimension at most $k$, with control only depending on $k,K$.
\end{corollary}

\subsection{Borel corollaries}

We recall the terminology needed for the Borel consequences. A \emphd{Polish space} is a separable completely metrizable topological space. A \emphd{standard Borel space} is a measurable space isomorphic to a Borel subset of a Polish space, equipped with its Borel sets. An action of a countable monoid on such a space is \emphd{Borel} if every map $x\mapsto ax$ is Borel. A \emphd{Borel graph} is a symmetric irreflexive Borel subset of $X\times X$. For a locally finite Borel graph (one with finitely many neighbors at each vertex), $\asdimB G\le n$ means that the $n+1$ color sets witnessing $\asdim G\le n$ can be chosen Borel at every scale. For a Borel action $M\curvearrowright X$ of a finitely generated
monoid and a finite generating set $S$, define
\[
    \asdimB(M\curvearrowright X)
    :=\asdimB(G_S^X).
\]


\begin{theorem}[{Shinko--Weilacher--Yu \cite[Corollary 5.16]{SWY}}]\label{cor:comm_monoid_asdim_eq}
    Let $M$ be a finitely generated commutative monoid and $X$ a bounded-to-one Borel $M$-set. 
    Then 
    \[
    \asdimB(M \curvearrowright X) = \asdim(M \curvearrowright X).
    \] 
\end{theorem}

\begin{corollary}\label{col:borel}
        Let $M\curvearrowright X$ be a bounded-to-one Borel action of a finitely generated commutative monoid on a standard Borel space. Then
\[
 \asdimB(M\curvearrowright X)\le\rk(M).
\]
\end{corollary}

For a countable commutative monoid action $M\curvearrowright X$, let $E_M^X$ be the equivalence relation generated by all pairs $(x,ax)$. Equivalently,
\[
 x\mathrel{E_M^X}y
 \quad\Longleftrightarrow\quad
 ax=by\text{ for some }a,b\in M.
\]
The dimension bound recovers the hyperfiniteness theorem of
Shinko--Weilacher--Yu stated above. 
We use the following two consequences of finite Borel asymptotic dimension.

\begin{theorem}[Conley--Jackson--Marks--Seward--Tucker-Drob {\cite[Theorem 1.7]{CJMST23}}]\label{thm:hyperfinite_of_dimension}
    Let $G$ be a locally finite Borel graph. If $\asdimB(G) < \infty$ then $E_G$ is hyperfinite.
\end{theorem}

\begin{theorem}[Conley--Jackson--Marks--Seward--Tucker-Drob {\cite[Theorem 1.10]{CJMST23}}]\label{thm:union}
    Let $G_0 \subseteq G_1 \subseteq \cdots$ be an increasing sequence of  locally finite Borel graphs on a standard Borel space with path metric $\rho_n$.  
    If $G_n$ has finite Borel asymptotic dimension for every $n$ then $E:= \bigcup_{n}E_{G_n}$ is hyperfinite.
\end{theorem}

\begin{corollary}[{Shinko--Weilacher--Yu \cite[Theorem 1.5]{SWY}}]\label{hyperfinite}
Every bounded-to-one Borel action of a countable commutative monoid has hyperfinite orbit equivalence relation.
\end{corollary}
\begin{proof}
Given a bounded-to-one action $M\curvearrowright X$,
    choose increasing finite sets $S_j\subseteq M$ with union $M$ and put $M_j=\langle S_j\rangle$. The locally finite Borel graphs $G_j=G_{S_j}^X$ are increasing, and Corollary~\ref{col:borel} gives $\asdimB(G_j) < \infty$. Then Theorem \ref{thm:union} implies that  $\bigcup_j E_{G_j}=E_M^X$ is hyperfinite.
\end{proof}

\subsection*{Note added}
The results of this paper were obtained independently of Ruijun Wang's work \cite{Wang};  A first version of our argument, formulated in terms of commuting maps, was completed in August 2026. Ruijun Wang proved the bound
\[
 \asdim(M\curvearrowright X)
 \le \frac{3^{\rk(M)+1}-3}{2}.
\]
Both proofs separate locally free points from relation sets, but use different geometric inputs: Wang uses greedy lattice markers and ball-multiplicity estimates, whereas we use a sharp dynamic asymptotic-dimension partition and color-preserving gluing.

\section{Algebraic preliminaries on monoids}\label{sec:pre}

We first recall the algebraic facts that will be used in the proof. 

\begin{definition}
A \emphd{monoid} is a set $M$ with an associative multiplication and an identity element $1$. It is \emphd{commutative} if $ab=ba$ for all $a,b\in M$. A \emphd{homomorphism} preserves multiplication and the identity. A \emphd{generating set} $S\subseteq M$ is a subset such that every element is a product of elements of $S$, where the empty product is $1$. The monoid is \emphd{finitely generated} if such a set can be chosen finite.
\end{definition}

All monoids below are commutative. For a finite generating set $S$, define the word length function
\[
 \ell_S(a)=\min\{j\geq0:a=s_1\cdots s_j,\ s_1,\ldots,s_j\in S\}.
\]
We may assume without loss of generality that $1\in S$.

\begin{definition}
An \emphd{action} $M\curvearrowright X$ assigns to every $a\in M$ a map $x\mapsto ax$ such that $1x=x$ and $(ab)x=a(bx)$. A subset $Y\subseteq X$ is \emphd{invariant} if $aY\subseteq Y$ for every $a\in M$. The action is \emphd{free} if $ax=bx$ implies $a=b$ for every $x\in X$. For a group action, this means that no nonidentity element fixes a point. A map is \emphd{bounded-to-one} if there is a finite upper bound for the cardinalities of its fibers. An action is bounded-to-one if every acting map has this property, with the bound allowed to depend on the monoid element.
\end{definition}

Throughout the paper, all monoid actions are assumed to be bounded-to-one.

If each generator is at most $K$-to-one, the map associated to $a$ is at most $K^{\ell_S(a)}$-to-one.

\begin{definition}[Congruences and quotient actions]
A \emphd{congruence} on $M$ is an equivalence relation $\theta$
compatible with multiplication:
\[
    a\mathrel{\theta} b
    \quad\Longrightarrow\quad
    ca\mathrel{\theta} cb
    \qquad(a,b,c\in M).
\]
The quotient $M/\theta$ is the monoid of $\theta$-equivalence
classes, with multiplication
\[
    [a]_\theta[b]_\theta=[ab]_\theta.
\]

Given $\mathcal R\subseteq M\times M$, the
\emphd{congruence generated by $\mathcal R$} is the smallest
congruence containing $\mathcal R$, equivalently the intersection
of all congruences containing $\mathcal R$. We write
\[
    \langle a=b\rangle
\]
for the congruence generated by the pair $(a,b)$.

An action of $M$ on a set $Y$ \emphd{factors through} $M/\theta$
if
\[
    ay=by
    \qquad
    \text{whenever }a\mathrel{\theta}b\text{ and }y\in Y.
\]
In that case the quotient action is defined by
\[
    [a]_\theta y=ay.
\]
\end{definition}

\begin{lemma}\label{lemma:relation}
For $a,b\in M$, the set
\[
 \period_{a,b}=\{x\in X:ax=bx\}
\]
is invariant, and its action factors through $M/\langle a=b\rangle$. Its generating maps are the restrictions of the original generating maps.
\end{lemma}
\begin{proof}
If $ax=bx$, then $a(cx)=c(ax)=c(bx)=b(cx)$ for every $c\in M$. Thus the set is invariant. The pairs of elements whose maps agree on this set form a congruence containing $(a,b)$, so they contain $\langle a=b\rangle$. This proves that the quotient action is well defined. The assertion about generating maps follows from the definition.
\end{proof}

\begin{definition}
A monoid is \emphd{cancellative} if $ac=bc$ implies $a=b$. A \emphd{group} is a monoid in which every element has an inverse. The \emphd{group completion} $\Gamma=M^{\gp}$ of a commutative monoid is the abelian group of formal fractions $[a,b]$, $a,b\in M$, with
\[
 [a,b]=[a',b']\quad\Longleftrightarrow\quad
 ab'c=a'bc\text{ for some }c\in M.
\]
Its multiplication, inverse, and natural homomorphism are
\[
 [a,b][a',b']=[aa',bb'],\qquad
 [a,b]^{-1}=[b,a],\qquad \iota(a)=[a,1].
\]
\end{definition}

\begin{lemma}\label{lemma:completion}
The map $\iota$ is injective if and only if $M$ is cancellative. If $M$ is finitely generated, then
\[
 \Gamma=M^{\gp}\cong\Z^n\oplus T
\]
for some finite abelian group $T$. We define $\rk(M)=n$.
\end{lemma}
\begin{proof}
The fraction construction gives
\[
 \iota(a)=\iota(b)\quad\Longleftrightarrow\quad
 ac=bc\text{ for some }c\in M,
\]
which proves the first assertion. The images of a finite generating set generate $M^{\gp}$ as a group. The structure theorem for finitely generated abelian groups gives the stated decomposition; see \cite[Section~2]{SWY}. Equivalently,
\[
 \rk(M)=\dim_{\Q}(M^{\gp}\otimes_{\Z}\Q).\qedhere
\]
\end{proof}

Only when $M$ is cancellative do we identify it with a submonoid of $\Gamma=M^{\gp}$. In that case we write $|u|_\Gamma$ for the word length in the group generators $S\cup S^{-1}$, and put
\[
 B_\Gamma(t)=\{u\in\Gamma:|u|_\Gamma\leq t\}.
\]
These sets are finite. The Cayley graph has edges $\{u,su\}$ for $s\in S$, and its distance is $d(u,v)=|u^{-1}v|_\Gamma$. We have $|a|_\Gamma\leq\ell_S(a)$ for $a\in M$; the two lengths need not agree.

\begin{lemma}\label{lemma:denominator}
Let $M$ be cancellative. For every finite set $E\subseteq\Gamma$, there is $q\in M$ such that $qE\subseteq M$.
\end{lemma}
\begin{proof}
Write $u=a_ub_u^{-1}$ with $a_u,b_u\in M$ for each $u\in E$. Take $q=\prod_{u\in E}b_u$. Then
\[
 qu=a_u\prod_{v\in E\setminus\{u\}}b_v\in M.
\]
For $E=\emptyset$, take $q=1$.
\end{proof}

We will use R\'edei's theorem in the following form \cite{Redei}; see also \cite[Theorem~2.1]{SWY}.
\begin{theorem}[R\'edei]\label{thm:redei}
Every congruence on a finitely generated commutative monoid is generated by finitely many pairs.
\end{theorem}

\begin{lemma}\label{lemma:cancellationAlgebra}
Let $M$ be a finitely generated commutative monoid, and let $N=\iota(M)\subseteq M^{\gp}$. Then $N$ is cancellative, $N^{\gp}=M^{\gp}$, and there exists $c\in M$ such that
\[
 \iota(a)=\iota(b)\quad\Longrightarrow\quad ac=bc
 \qquad(a,b\in M).
\]
\end{lemma}
\begin{proof}
The relation $\iota(a)=\iota(b)$ is a congruence. By Theorem~\ref{thm:redei}, it is generated by finitely many pairs $(a_j,b_j)$. For each pair, the fraction construction supplies $c_j\in M$ with $a_jc_j=b_jc_j$. Put $c=\prod_j c_j$, using $c=1$ if there are no pairs. The relation $ac=bc$ is itself a congruence, by commutativity, and contains every generating pair. It therefore contains the kernel congruence of $\iota$, proving the assertion. The monoid $N$ is a submonoid of a group, hence cancellative, and its fractions are exactly the elements of $M^{\gp}$.
\end{proof}

\begin{lemma}\label{lemma:powersAlgebra}
Let $M$ be cancellative, with finite generating set $S=\{s_1,\ldots,s_k\}$. Write $\Gamma=M^{\gp}\cong\Z^n\oplus T$ as in Lemma~\ref{lemma:completion}. Choose $p\geq1$ such that $t^p=1$ for every $t\in T$. The submonoid
\[
 N=\langle s_1^p,\ldots,s_k^p\rangle\subseteq M
\]
is cancellative and has torsion-free group completion of rank $n$.
\end{lemma}
\begin{proof}
Its group completion identifies with the subgroup generated by the elements $s_j^p$ in $\Gamma$. Since $S$ generates $\Gamma$ as a group, this is the image of the homomorphism $u\mapsto u^p$. Under $\Gamma\cong\Z^n\oplus T$, this image is $p\Z^n\oplus\{1\}$, which is torsion-free of rank $n$. Cancellativity follows from $N\subseteq\Gamma$.
\end{proof}

\begin{lemma}\label{lemma:rankDrop}
Let $M$ be cancellative with torsion-free group completion of rank $n$. If $a,b\in M$ are distinct, then
\[
 \rk\bigl(M/\langle a=b\rangle\bigr)=n-1.
\]
\end{lemma}
\begin{proof}
The group completion of $M/\langle a=b\rangle$ is
\[
 M^{\gp}/\langle ab^{-1}\rangle.
\]
Indeed, homomorphisms from either group to an abelian group correspond exactly to homomorphisms from $M$ identifying $a$ and $b$, by the defining property of group completion and of the generated congruence. Since $M$ embeds in its torsion-free completion, $ab^{-1}\ne1$ has infinite order. Quotienting a finitely generated abelian group by this cyclic subgroup lowers its rank by exactly one.
\end{proof}

\section{Outline of the proof}\label{sec:decomposition}

The proof has two main steps. We first reduce to the situation where the monoid is cancellative and has torsion-free group completion. We then divide the space into a period part, where a short relation holds, and a free part, where no such relation holds. 
In this section, we explain these steps and how their estimates fit together. The reduction is proved in Section~\ref{sec:reduction},  the period part estimate in Section~\ref{sec:period}, and the free part estimate in Section~\ref{sec:free}. 

We begin by fixing the metric terminology used in the argument.

Let $M\curvearrowright X$ be an action and let $S$ be a finite generating set. We define edges $\{x,sx\}$, $x\in X$, $s\in S$, and denote the graph metric by $d$. Loops are discarded. The distance between points in different connected components is $\infty$. All subspace distances below are restrictions of this metric, unless another metric is specified.

\begin{definition}\label{def:asymDim}
Let $(X,d)$ be a metric space, possibly with infinite distances. An $r$-chain in a subset is a finite sequence of its points with successive distances at most $r$. Its $r$-connected components are the equivalence classes generated by these chains. We say that $(X,d)$ has \emphd{asymptotic dimension at most $n$} if for every $r>0$ there is a cover $U_0,\ldots,U_n$ of $X$ such that every $r$-connected component of each $U_i$ has diameter at most some $R(r)<\infty$. We call $R$ a \emphd{control function}. At a single prescribed scale, we say that the space has asymptotic dimension at most $n$ at scale $r$ with control $R(r)$.
\end{definition}

A cover can be made a partition by assigning each point to its first member. Empty members are allowed. 
We write $\asdim(M\curvearrowright X)$ for the asymptotic dimension of this graph. 

\subsection{Reduction to a cancellative monoid with torsion-free completion}

We first arrange that the monoid embeds in a free abelian group. The algebraic results of Section~2 suggest two operations: pass to the cancellative image $\iota(M)$, and then take powers of its generators to remove torsion from the completion. For an action on $X$, the first operation is carried out on a suitable invariant image $\mathcal T=cX$. Every point of $X$ lies a bounded distance from this image. The next lemma states that these two operations preserve exactly the asymptotic dimension needed in the proof.

\begin{lemma}\label{lemma:reduction}
Let $M$ be a finitely generated commutative monoid, with generating set $S=\{s_1,\ldots,s_k\}$ and natural map $\iota\colon M\to M^{\gp}$. Choose $p\geq1$ annihilating the torsion subgroup of $M^{\gp}$, and put
\[
 N=\langle\iota(s_1)^p,\ldots,\iota(s_k)^p\rangle\subseteq M^{\gp}.
\]
There exists $c\in M$, depending only on $M$, such that every action $M\curvearrowright X$ induces an action $N\curvearrowright\mathcal T$ on $\mathcal T=cX$ and
\[
 \asdim(M\curvearrowright X)=\asdim(N\curvearrowright\mathcal T).
\]
The monoid $N$ is cancellative, its group completion is torsion-free, and $\rk(N)=\rk(M)$. If the generators in $S$ have fiber bound $K$, the displayed generators of $N$ have fiber bound $K^p$ on $\mathcal T$. For a fixed dimension bound, controls transfer in both directions with dependence only on $M,S,p$ and the given control.
\end{lemma}

We prove this lemma in Section \ref{sec:reduction}.

\subsection{The free-period decomposition and the induction}

Assume now that $M$ is cancellative, that $\Gamma=M^{\gp}$ is torsion-free, and that $n=\rk(M)$. Fix a finite generating set $S$ and a common fiber bound $K$ for its generators. The induction starts at $n=0$: the reduced monoid is then the identity monoid, and every graph component is a singleton. For $n\geq1$, we assume the theorem, with uniform controls, for every monoid of smaller rank.

For $a,b\in M$, recall that
\[
 \period_{a,b}=\{x\in X:ax=bx\}.
\]
 For an integer $C\geq1$, let
\[
 \period_C=\bigcup_{\substack{a,b\in M,\ a\ne b\\
                  \ell_S(a)+\ell_S(b)\leq C}}\period_{a,b},
 \qquad \free_C=X\setminus\period_C.
\]
Thus $\period_C$ is the $C$-period part and $\free_C$ is the $C$-free part. Only finitely many pairs occur in this union. At a point of $\free_C$, evaluation distinguishes any two monoid elements whose total positive word length is at most $C$.

On a relation set $\period_{a,b}$, the action factors through $M/\langle a=b\rangle$. Lemma~\ref{lemma:rankDrop} lowers its rank from $n$ to $n-1$. Applying the induction hypothesis to these finitely many quotients gives the period estimate. 

\begin{lemma}\label{lemma:period}
With the notation above, assume that Theorem~\ref{thm:rank}, including its assertion about controls, holds for all finitely generated commutative monoids of rank $<n$. Then for every action $M\curvearrowright X$ whose generators in $S$ have fiber cardinalities at most $K$, and every integer $C\geq1$, the set $\period_C$ has asymptotic dimension at most $(n-1)$, with control only depending on $M,S,K,C$.
\end{lemma}

Its proof will be given in Section~\ref{sec:period}.

\medskip

The number $C$ above is produced from the free part after fixing $r$.

\begin{lemma}\label{lemma:free}
Let $M$ be cancellative with torsion-free group completion of rank $n$, let $S$ be a finite generating set, and let $K\geq1$ be an integer. For every integer $r\geq1$, there exists $C=C(M,S,K,r)\geq1$ such that, for every action $M\curvearrowright X$ whose generators in $S$ have fiber cardinalities at most $K$, the set $\free_C$, with respect to the metric from $X$, has asymptotic dimension at most $n$ at scale $r$, with control $R_\free$ only depending on $M,S,K,r$.
\end{lemma}

Its proof will be given in Section~\ref{sec:free}.

\medskip

These two estimates have different roles. The free estimate is obtained at the prescribed scale $r$, with a resulting cutoff $C$ and bound $R_\free$. Once $C$ is fixed, the period estimate is available at every scale. We therefore use it at the larger scale
\[
 r'=R_\free+2r.
\]
This choice accounts for an $r$-chain in one color which leaves the period part, travels through one $r$-component of that color in the free part, and returns. The distance between its two visits to the period part is at most $r'$. The following elementary gluing statement makes this observation precise.

\begin{lemma}\label{lemma:asdimDecompos}
Suppose $X=A\cup B$. Let $U_0,\ldots,U_n$ cover $A$, with every $r$-component of each $U_i$ of diameter at most $R_\free$. Let $W_0,\ldots,W_n$ cover $B$, with every $(R_\free+2r)$-component of each $W_i$ of diameter at most $R_\period$. Then $V_i=U_i\cup W_i$ cover $X$, and every $r$-component of each $V_i$ has diameter at most
\[
 R_\period+2R_\free+2r.
\]
\end{lemma}
\begin{proof}
Let $Z$ be an $r$-component of $V_i$. If $Z\cap W_i=\emptyset$, then $Z$ lies in one $r$-component of $U_i$, and the conclusion follows.

Otherwise, consider any $r$-chain in $Z$ between two points of $W_i$. Between two successive visits to $W_i$, the chain either has one step, or passes through a single $r$-component of $U_i$. The distance between those two visits is therefore at most $R_\free+2r$. It follows that $Z\cap W_i$ lies in one $(R_\free+2r)$-component of $W_i$, so its diameter is at most $R_\period$. Any point of $Z$ is within $R_\free+r$ of $Z\cap W_i$, by considering its first visit to $W_i$. The triangle inequality proves the bound.
\end{proof}

The idea of the proof of Lemma \ref{lemma:asdimDecompos} will also be repeatedly used in the rest of the paper. With all the ingredients in this section, we can now prove Theorem \ref{thm:rank}.

\begin{proof}[Proof of Theorem~\ref{thm:rank}]
By Lemma \ref{lemma:reduction}, it suffices to consider cancellative monoids with torsion-free group completion of the same rank. We prove by induction on the rank $n$. The assertion for $n=0$ is trivial.

Suppose $n\geq1$, and assume the theorem, including the assertion about controls. Fix an integer $r\geq1$. Lemma~\ref{lemma:free} gives $C$ and a partition $U_0,\ldots,U_n$ of $\free_C$ whose $r$-components are $R_\free$-bounded. Apply Lemma~\ref{lemma:period} at scale
\[
 r'=R_\free+2r.
\]
It gives a partition $W_0,\ldots,W_{n-1}$ of $\period_C$ whose $r'$-components are $R_\period$-bounded. Put $W_n=\emptyset$ and define
\[
 V_i=U_i\cup W_i\qquad(0\leq i\leq n).
\]
Lemma~\ref{lemma:asdimDecompos} shows that their $r$-components have diameter at most
\[
 R_\period+2R_\free+2r.
\]
The controls in the period and free propositions depend only on $M,S,K$ and the indicated scales. Thus this bound is uniform over all the actions under consideration. The reductions retain this uniformity for the original monoid, which completes the induction.

Finally, the translation action $\N^n\curvearrowright\Z^n$ has graph the standard lattice, whose asymptotic dimension is $n$ \cite{Roe}. Thus the dimension bound is sharp.
\end{proof}
\section{Reduction}\label{sec:reduction}

In this section, we prove the reduction Lemma \ref{lemma:reduction}, which reduces an arbitrary monoid action to a cancellative monoid with torsion free group completion.

We first give an equivalent characterization of the graph distance.
\begin{lemma}\label{lemma:uniqueSink}
For every $x,y\in X$,
\[
 d(x,y)=\min\{\ell_S(a)+\ell_S(b):a,b\in M,\ ax=by\},
\]
where the minimum of the empty set is $\infty$.
\end{lemma}
\begin{proof}
An equality $ax=by$ gives a path consisting of a forward path from $x$ to $ax$ and a backward path from $by$ to $y$. Its length is at most $\ell_S(a)+\ell_S(b)$.

Conversely, each edge has such a description with total word length at most one. Suppose $ax=bz$ and $cz=ey$ describe two successive portions of a path. Commutativity gives
\[
 (ca)x=c(ax)=c(bz)=b(cz)=(be)y,
\]
and the new total word length is at most the sum of the two previous lengths. Induction on path length proves the converse inequality. This is the common-sink form of a path: all forward steps can be placed before all backward steps.
\end{proof}

\begin{corollary}\label{cor:invariant}
For an invariant subset $Y\subseteq X$, its own graph metric agrees with the ambient metric. For $c\in M$, the set $cX$ is invariant and
\[
 d(x,cx)\leq\ell_S(c)\qquad(x\in X).
\]
\end{corollary}
\begin{proof}
Both forward paths in Lemma~\ref{lemma:uniqueSink} stay in $Y$ when their initial points lie in $Y$. This proves equality of the metrics. Commutativity gives $a(cX)\subseteq cX$, and a word representing $c$ gives the displayed bound.
\end{proof}

Every acting map is $1$-Lipschitz for $d$, since it sends an edge to an edge or a point. If $S'$ is another finite generating set, replacing every $S'$-edge by an $S$-word, and conversely, compares the two graph metrics by constant factors. Thus $\asdim(M\curvearrowright X)$ is independent of the finite generating set.

Repeated use of Lemma~\ref{lemma:asdimDecompos} gives the finite union theorem \cite{BellDranishnikov}: a finite union of subspaces of asymptotic dimension at most $n$ has the same bound, with control depending only on their controls and their number.

\medskip

To prove Lemma \ref{lemma:reduction}, we first pass to a forward image at uniformly bounded distance from $X$, on which the action factors through the cancellative monoid. We then compare the original generators with their powers.

\begin{lemma}\label{lemma:cancellationMetric}
Let $M$ be a finitely generated commutative monoid and let $N=\iota(M)\subseteq M^{\gp}$. There is $c\in M$, depending only on $M$, such that for every action $M\curvearrowright X$ the action on $\mathcal T=cX$ factors through $N$. With the image of $S$ as generating set,
\[
 \asdim(M\curvearrowright X)=\asdim(N\curvearrowright\mathcal T).
\]
If $R$ is a control for the latter action, a control for the former is
\[
 r\longmapsto R(r+2\ell_S(c))+2\ell_S(c).
\]
The generator fiber bound does not increase.
\end{lemma}
\begin{proof}
Choose $c$ from Lemma~\ref{lemma:cancellationAlgebra}. If $\iota(a)=\iota(b)$, then $ac=bc$, and hence $ay=by$ for every $y\in\mathcal T=cX$. Thus the action on this invariant set factors through $N$. The quotient generating maps are the restrictions of the original generating maps, so their fiber bounds do not increase. By Corollary~\ref{cor:invariant}, the metric on $\mathcal T$ is the ambient metric and $d(x,cx)\leq\ell_S(c)$. The bounded-displacement comparison gives the asserted control. Restricting covers of $X$ to $\mathcal T$ gives the reverse dimension inequality.
\end{proof}

A metric space is \emphd{proper} if its closed bounded balls are compact. The graph components below are proper whenever all finite-radius balls are finite.

\begin{lemma}\label{lemma:finiteQuotient}
Let $Y$ be an extended metric space whose finite-distance components are proper, and let a finite group $F$ act by isometries. Give the orbit space $Y/F$ the metric
\[
 d_q([x],[y])=\min_{h\in F}d_Y(x,hy).
\]
Then
\[
 \asdim(Y)=\asdim(Y/F).
\]
For a fixed dimension bound $n$, controls transfer in both directions with dependence only on $n$, the given control, and $|F|$.
\end{lemma}
\begin{proof}
First suppose that $Y$ has an $n$-dimensional control. Isometries permute the finite-distance components. If $Z$ is one such component and $H=\{h\in F:hZ=Z\}$, the corresponding quotient component is isometric to $Z/H$. Indeed, it has representatives in $Z$, and for $x,y\in Z$ the distance $d_Y(x,hy)$ is finite exactly when $h\in H$. The components $Z$ have the same control, are proper, and satisfy $|H|\leq|F|$. Now it follows from \cite[Corollary~1.2]{Kasprowski}.

For the reverse inequality, let $R$ be an $n$-dimensional control for $Y/F$, and pull back its cover at scale $r$ under the quotient map. An $r$-component $Z$ of a pulled-back member projects into a single $r$-component of that member. Fix $z\in Z$. By the quotient metric, $Z$ is contained in the union of the $|F|$ balls of radius $R(r)$ centered at the points $hz$, $h\in F$. Along an $r$-chain, centers of successive balls are at distance at most $2R(r)+r$. A simple path in the graph of these balls uses at most $|F|-1$ transitions. Estimating the two endpoint distances to their centers therefore gives
\[
 \operatorname{diam}(Z)\leq 2|F|R(r)+(|F|-1)r.
\]
This proves the reverse inequality with the required control, including when $Y$ has infinite distances.
\end{proof}

\begin{lemma}\label{lemma:powers}
Let $M\curvearrowright X$ be a bounded-to-one action of a commutative monoid generated by $S=\{s_1,\ldots,s_k\}$. Fix an integer $p\geq1$. Write $d$ for the graph metric given by $S$, and $d_p$ for the graph metric given by $s_1^p,\ldots,s_k^p$. Then
\[
 \asdim(X,d)=\asdim(X,d_p).
\]
For a fixed dimension bound $n$, controls transfer in both directions with dependence only on $n,p,k$ and the given control.
\end{lemma}
\begin{proof}
First suppose that $(X,d_p)$ has an $n$-dimensional control. Let
\[
 I=\{0,\ldots,p-1\}^k,\qquad F=(\Z/p\Z)^k.
\]
We use $I$ as the set of representatives of $F$. Equip $X^I$ with the extended metric
\[
 d_\infty(\underline x,\underline y)
 =\max_{a\in I}d_p(x_a,y_a).
\]
Define
\[
 \Phi\colon X\longrightarrow X^I,\qquad
 \Phi(x)_a=s_1^{a_1}\cdots s_k^{a_k}x.
\]
Every map associated to an element of $M$ is $1$-Lipschitz for $d_p$, because it commutes with all the maps $s_j^p$. The coordinate $a=0$ is $x$ itself. Consequently,
\begin{equation}\label{eq:powersIsometry}
 d_\infty(\Phi(x),\Phi(y))=d_p(x,y).
\end{equation}

For $b\in F$, let $\tau_b$ permute the coordinates by
\[
 (\tau_b\underline x)_a=x_{a+b},
\]
where addition in the coordinate subscript is taken modulo $p$. These coordinate permutations form an isometric action of $F$. Put
\[
 \mathcal A=\bigcup_{b\in F}\tau_b\Phi(X).
\]
Thus $F$ acts isometrically on $\mathcal A$. Each of the $p^k$ subspaces in this union is isometric to $(X,d_p)$ by \eqref{eq:powersIsometry}. The finite union theorem therefore gives asymptotic dimension at most $n$ for $\mathcal A$, with control depending only on the input control and $p^k$.

We check properness. The maps $s_j^p$ are bounded-to-one, so the graph defining $d_p$ has bounded degree and hence finite balls of finite radius. Consider a ball of radius $R$ in $\mathcal A$. For every $b\in F$ whose subspace $\tau_b\Phi(X)$ meets this ball, choose one point of the intersection. The intersection is contained in the radius-$2R$ ball about that point inside $\tau_b\Phi(X)$. This ball is finite by \eqref{eq:powersIsometry}. There are at most $p^k$ subspaces, so the original ball is finite. In particular, every finite-distance component of $\mathcal A$ is proper.

Let $\pi\colon\mathcal A\to\mathcal A/F$ be the quotient map, and give the quotient the metric
\[
 d_q([\underline x],[\underline y])
 =\min_{b\in F}d_\infty(\underline x,\tau_b\underline y).
\]
The componentwise finite-quotient theorem, Lemma~\ref{lemma:finiteQuotient}, gives asymptotic dimension at most $n$ for $\mathcal A/F$. Its control $R_q$ only depends on $n$, the control of $\mathcal A$, and $p^k$.

We compare this metric with $d$. Fix $j\in\{1,\ldots,k\}$ and compare $\Phi(s_jx)$ with $\tau_{e_j}\Phi(x)$, where $e_j$ is the $j$th standard generator of $F$. At a coordinate $a$ with $a_j<p-1$, these tuples agree. At a coordinate with $a_j=p-1$, the first entry is the image of the second under $s_j^p$. Their $d_p$-distance is therefore at most one. This proves
\[
 d_q(\pi\Phi(x),\pi\Phi(s_jx))\leq1.
\]
It follows along paths that
\begin{equation}\label{eq:powersUpper}
 d_q(\pi\Phi(x),\pi\Phi(y))\leq d(x,y).
\end{equation}

Conversely, suppose the left side is finite, and let $b\in I$ represent a group element attaining the minimum in its definition. The coordinate $a=0$ gives
\[
 d_p\bigl(x,s_1^{b_1}\cdots s_k^{b_k}y\bigr)
 \leq d_q(\pi\Phi(x),\pi\Phi(y)).
\]
Every edge for $d_p$ is a path of length at most $p$ for $d$. Since $0\leq b_j\leq p-1$, we obtain
\begin{equation}\label{eq:powersLower}
 d(x,y)\leq p\,d_q(\pi\Phi(x),\pi\Phi(y))+k(p-1).
\end{equation}
For an $r$-chain in $X$, \eqref{eq:powersUpper} makes its image an $r$-chain in the quotient. Pulling back a cover with control $R_q$, and applying \eqref{eq:powersLower} to endpoints, gives the control
\[
 r\longmapsto pR_q(r)+k(p-1).
\]
This proves $\asdim(X,d)\leq\asdim(X,d_p)$, with the asserted uniformity.

For the reverse inequality, the map $\pi\Phi\colon X\to\mathcal A/F$ is onto, since every orbit in $\mathcal A$ contains a point of $\Phi(X)$. Choose one preimage in $X$ for each quotient point. By \eqref{eq:powersLower}, these choices send an $r$-chain in the quotient to a $(pr+k(p-1))$-chain in $(X,d)$. By \eqref{eq:powersUpper}, they give a cover of the quotient with control
\[
 r\longmapsto R(pr+k(p-1))
\]
whenever $R$ is a control for $(X,d)$. Lemma~\ref{lemma:finiteQuotient} now gives a control for $\mathcal A$. Restricting that cover to its isometric subspace $\Phi(X)$, and using \eqref{eq:powersIsometry}, gives a control for $(X,d_p)$. Each transfer depends only on $n,p,k$ and $R$, completing the proof.
\end{proof}

\begin{proof}[Proof of Lemma~\ref{lemma:reduction}]
Choose $c$ as in Lemma~\ref{lemma:cancellationAlgebra}. Lemma~\ref{lemma:cancellationMetric} gives an action of $\iota(M)$ on $\mathcal T=cX$ and the equality
\[
 \asdim(M\curvearrowright X)
 =\asdim(\iota(M)\curvearrowright\mathcal T).
\]
Restrict this action to $N$. Its generating maps are the restrictions of $s_1^p,\ldots,s_k^p$, so their fiber bound is $K^p$. Lemma~\ref{lemma:powersAlgebra}, applied to $\iota(M)$, gives the assertions about cancellation, torsion, and rank. Lemma~\ref{lemma:powers} gives
\[
 \asdim(\iota(M)\curvearrowright\mathcal T)
 =\asdim(N\curvearrowright\mathcal T).
\]
The control assertions in the two metric lemmas give the stated uniformity.
\end{proof}
\section{The period part}\label{sec:period}

In this section, we work with the period part in Lemma \ref{lemma:period}. Roughly speaking, each relation set carries an action of a quotient of strictly smaller rank.

\begin{proof}[Proof of Lemma~\ref{lemma:period}]
Fix $C\geq1$. For each distinct pair $a,b$ with $\ell_S(a)+\ell_S(b)\leq C$, Lemma~\ref{lemma:relation} gives an action of
\[
 N=M/\langle a=b\rangle
\]
on $\period_{a,b}$. Its generating set is the image of $S$, and its generating maps have fiber bound at most $K$. By Lemma~\ref{lemma:rankDrop}, $\rk(N)=n-1$. The induction hypothesis therefore gives asymptotic dimension at most $n-1$ on $\period_{a,b}$, with control depending only on this quotient, the image of $S$, and $K$. The quotient need not be cancellative or have torsion-free completion; the induction hypothesis applies to every monoid of this smaller rank.

The intrinsic graph metric on $\period_{a,b}$ agrees with its ambient metric by Corollary~\ref{cor:invariant}. Since $M,S,C$ are fixed, only finitely many pairs occur, and the associated quotient monoids are fixed independently of the action. The finite union theorem consequently gives the required bound on $\period_C$.

For completeness, the uniform control in this finite union can be read directly from Lemma~\ref{lemma:asdimDecompos}. If the relation sets have controls $H_1,\ldots,H_q$, define
\[
 D_1(r)=H_1(r),
\]
\[
 D_{j+1}(r)=H_{j+1}(D_j(r)+2r)+2D_j(r)+2r.
\]
Then $D_q$ is a control for their union. Empty relation sets may be discarded; if all are empty there is nothing to prove. The resulting function depends only on $M,S,K,C$.
\end{proof}

\section{The free part}\label{sec:free}

In this section, we work with the free part in Lemma \ref{lemma:free}.

\subsection{Universal space}
Fix an integer $K\geq1$, and put
\[
 \Delta=\max\{2,|S|(K+1)\},\qquad
 \Lambda_m=\left\{1,\ldots,1+\sum_{j=1}^m\Delta^j\right\}
 \quad(m\geq1).
\]
The graph of an action with generator fiber bound $K$ has degree at most $|S|(K+1)$. Its $m$th power, which joins distinct vertices at distance at most $m$, therefore has degree at most $\sum_{j=1}^m\Delta^j$. We can choose a coloring
\begin{equation}\label{eq:lambda}
 \lambda_m\colon X\longrightarrow\Lambda_m,\qquad
 0<d(x,y)\leq m\ \Longrightarrow\ \lambda_m(x)\ne\lambda_m(y).
\end{equation}
Indeed, every component is countable, and greedy coloring of its enumerated vertices uses at most one more color than the degree. The same alphabets also color the powers of the Cayley graph of $\Gamma$, whose degree is at most $2|S|\leq\Delta$. The alphabets are fixed independently of the action and of the chosen colorings.

Let
\[
 \Lambda=\prod_{m\geq1}\Lambda_m,\qquad \Omega=\Lambda^\Gamma,
\]
with the product topology, each $\Lambda_m$ being discrete. Write $z_{u,m}$ for the coordinates of $z\in\Omega$. The group acts by shifts
\begin{equation}\label{eq:shift}
 (\sigma_s z)_{u,m}=z_{us^{-1},m}.
\end{equation}
Thus $\sigma_1=\mathrm{id}$ and $\sigma_s\sigma_t=\sigma_{st}$. Define
\begin{equation}\label{eq:universal}
 \univ=\bigcap_{m\geq1}\ \bigcap_{0<|u^{-1}v|_\Gamma\leq m}
 \{z\in\Omega:z_{u,m}\ne z_{v,m}\}.
\end{equation}
A cylinder prescribes finitely many coordinates in $\Omega$. Such sets are both clopen. A space is called \emphd{zero-dimensional} if it has a basis of clopen sets.

\begin{lemma}\label{lemma:universal}
The space $\univ$ is nonempty, compact, metrizable, zero-dimensional, and invariant under $\sigma$. The $\Gamma$-action on $\univ$ is free. Moreover, for every finite set $E\subseteq\Gamma\times\{1,2,\ldots\}$ and every $a\in\prod_{(u,m)\in E}\Lambda_m$, we have
\[
 \univ\cap\{z\in\Omega:z_{u,m}=a_{u,m}\text{ for all }(u,m)\in E\}
 \ne\emptyset
\]
if and only if
\[
 a_{u,m}\ne a_{v,m}
 \quad\text{whenever }(u,m),(v,m)\in E,
 \quad 0<|u^{-1}v|_\Gamma\leq m.
\]
\end{lemma}
\begin{proof}
The space $\Omega$ is a countable product of finite discrete spaces, hence compact and metrizable, with a basis of clopen cylinders. The conditions in \eqref{eq:universal} are closed, so $\univ$ is compact, metrizable, and zero-dimensional. A proper $\Lambda_m$-coloring of the $m$th Cayley-graph power exists by the preceding degree estimate. Choosing such a coloring in each layer gives a point of $\univ$.

Each shift is a homeomorphism of $\Omega$, since it permutes the coordinates. The defining conditions are invariant under shifts. If $\sigma_s z=z$ with $s\ne1$, choose $m=|s|_\Gamma\geq1$. Then
\[
 z_{1,m}=(\sigma_s z)_{1,m}=z_{s^{-1},m},
\]
contrary to \eqref{eq:universal}. Thus the action is free.

For the last assertion, necessity follows from \eqref{eq:universal}. Conversely, suppose that $a$ satisfies the displayed inequalities. In layer $m$, keep the colors $a_{u,m}$ for $(u,m)\in E$ and enumerate the remaining group elements. Greedily assign colors to them. At most $\sum_{j=1}^m\Delta^j$ colors are forbidden at each step, including the colors of all prescribed neighbors, so a color is available. For a layer without a prescription, choose any proper coloring. The layers have no additional constraints between them. The resulting point of $\univ$ has the prescribed coordinates, proving sufficiency.
\end{proof}

Guentner, Willett, and Yu introduced dynamic asymptotic dimension in \cite{GWY}. It adapts asymptotic dimension to topological actions by using open covers to break orbit chains into uniformly controlled pieces. Its original motivations came from controlled topology and operator $K$-theory. We recall the definition here \cite[Definition~2.1]{GWY}.

\begin{definition}\label{def:dad}
Let $\Gamma\curvearrowright Y$ be an action by homeomorphisms on a compact metrizable space. We say that
\[
 \operatorname{dad}(\Gamma\curvearrowright Y)\leq n
\]
if, for every finite symmetric set $E\subseteq\Gamma$ containing $1$, there are an open cover $W_0,\ldots,W_n$ of $Y$ and a finite set $F\subseteq\Gamma$ with the following property. For every $i$, $y\in Y$, and finite sequence $s_0=1,s_1,\ldots,s_p\in\Gamma$,
\[
 s_js_{j-1}^{-1}\in E\quad(1\leq j\leq p),
 \qquad s_jy\in W_i\quad(0\leq j\leq p)
\]
imply $s_p\in F$. The least such $n\in\N$ is the \emphd{dynamic asymptotic dimension} of the action; it is $\infty$ if no such $n$ exists.
\end{definition}

Combining \cite[Theorem 4.11 \& Remark 4.14]{GWY} and
\cite[Lemma 8.4]{SzaboWuZach}, the dynamic asymptotic dimension of the amenable group $\Z^n$ freely acting on a zero dimension space is bounded by a sub-optimal constant in $n$. It is further strengthened by \cite[Corollary~3.5]{Pilgrim} and \cite[Theorem~1.3]{CJMST23} that the dynamic asymptotic dimension is precisely bounded by its rank $n$. More precisely, the optimal dimension bound for our use is stated as follows.

\begin{theorem}\label{thm:dadZn}
If $\Gamma\cong\Z^n$ and $Y$ is a compact, metrizable, zero-dimensional space on which $\Gamma$ acts freely by homeomorphisms, then
\[
 \operatorname{dad}(\Gamma\curvearrowright Y)\leq n.
\]
\end{theorem}

Applying this theorem to the universal space gives the following form of the dynamical input for the free part.
\begin{proposition}\label{prop:dad}
For every integer $r\geq1$, there are $D\geq0$ and a clopen partition
\[
 \univ=W_0\sqcup\cdots\sqcup W_n
\]
such that, whenever $s_0=1,s_1,\ldots,s_p\in\Gamma$ satisfy
\[
 |s_js_{j-1}^{-1}|_\Gamma\leq r\qquad(1\leq j\leq p)
\]
and $\sigma_{s_j}(z)\in W_i$ for all $0\leq j\leq p$, for some $z\in\univ$, then $|s_p|_\Gamma\leq D$.
\end{proposition}
\begin{proof}
By Lemma~\ref{lemma:universal}, the action $\Gamma\curvearrowright\univ$ satisfies the hypotheses of Theorem~\ref{thm:dadZn}. Apply Definition~\ref{def:dad} with $E=B_\Gamma(r)$. It gives an open cover by $n+1$ sets and a finite displacement set $F$. Choose an integer $D\geq0$ such that $F\subseteq B_\Gamma(D)$.

By zero-dimensionality and compactness, choose finitely many clopen sets refining the open cover and covering $\univ$. Taking their successive differences gives a finite clopen partition still refining that cover. Assign each atom to an original cover member containing it and combine atoms with the same assignment. This gives $W_0,\ldots,W_n$ without changing the displacement bound.
\end{proof}

For $L\geq1$, let $P_L$ denote the projection onto the coordinates
\[
 B_\Gamma(L)\times\{1,\ldots,L\}.
\]
\begin{lemma}\label{lemma:finiteDependence}
There is an integer $L\geq1$ such that membership in the partition $W_0,\ldots,W_n$ is determined by $P_L$: if $z,z'\in\univ$ and $P_L(z)=P_L(z')$, then they belong to the same $W_i$.
\end{lemma}
\begin{proof}
For every $z\in\univ$, choose a cylinder in $\Omega$ containing $z$ whose intersection with $\univ$ lies in the member containing $z$. Finitely many such cylinders cover $\univ$. Their defining coordinates lie in $B_\Gamma(L)\times\{1,\ldots,L\}$ for some $L$. Two points agreeing there belong to the same selected cylinders, and hence to the same partition member.
\end{proof}

\subsection{Partition for the free part}
We now relate $\free_C$ to the universal space. The action on $X$ need not extend to $\Gamma$, so we first move the group coordinates used in the construction into $M$ by a common denominator.

\begin{proof}[Proof of Lemma~\ref{lemma:free}]
Fix $M,S,K,r$. Choose $\univ$, then $W_0,\ldots,W_n$ and $D$ by Proposition~\ref{prop:dad}, and finally $L$ by Lemma~\ref{lemma:finiteDependence}. All these choices are independent of the action on $X$.

Put
\[
 w=L+D+r.
\]
By Lemma~\ref{lemma:denominator}, choose $q\in M$ such that
\begin{equation}\label{eq:buffer}
 qB_\Gamma(2w)\subseteq M.
\end{equation}
Choose an integer $C\geq1$ with
\begin{equation}\label{eq:C}
 C\geq2\max_{u\in B_\Gamma(w)}\ell_S(qu).
\end{equation}
The maximum is finite. Consequently, if $x\in\free_C$, the points $(qu)x$, $u\in B_\Gamma(w)$, are pairwise distinct: distinct $u$ give distinct $qu$ in $M$, and their positive lengths have sum at most $C$.

Now fix an action with generator fiber bound $K$ and choose the colorings \eqref{eq:lambda}. For $x\in\free_C$, define
\begin{equation}\label{eq:E}
 \mathcal E(x,w)=\{z\in\Omega:
 z_{u,m}=\lambda_m((qu)x),\ |u|_\Gamma\leq w,\ 1\leq m\leq L\}.
\end{equation}
We claim that
\begin{equation}\label{eq:extension}
 \mathcal E(x,w)\cap\univ\ne\emptyset.
\end{equation}
Suppose $u,v\in B_\Gamma(w)$ and $0<|u^{-1}v|_\Gamma\leq m\leq L$. A shortest Cayley path between them has length at most $m$ and lies in
\[
 B_\Gamma(w+L)\subseteq B_\Gamma(2w).
\]
Multiplication by $q$ moves its vertices into $M$ by \eqref{eq:buffer}. If two successive vertices satisfy $v'=sv$ or $v=sv'$, their translates satisfy the same equality in $M$, since $M\hookrightarrow\Gamma$. Evaluation at $x$ therefore gives a path of length at most $m$ in $X$. Thus
\[
 d((qu)x,(qv)x)\leq m.
\]
The endpoints are distinct by \eqref{eq:C}, so their $\lambda_m$-colors differ. Each prescribed layer in \eqref{eq:E} is a proper partial coloring. Lemma~\ref{lemma:universal} proves \eqref{eq:extension}.

Every point of $\mathcal E(x,w)\cap\univ$ has the same $P_L$-projection. By Lemma~\ref{lemma:finiteDependence}, this nonempty set is contained in one $W_i$. Define
\[
 V_i=\{x\in\free_C:\mathcal E(x,w)\cap\univ\subseteq W_i\},
 \qquad 0\leq i\leq n.
\]
These sets form a partition of $\free_C$.

Let $x_0,\ldots,x_p$ be an $r$-chain in one $V_i$. By Lemma~\ref{lemma:uniqueSink}, choose $a_j,b_j\in M$ such that
\begin{equation}\label{eq:chain}
 a_jx_{j-1}=b_jx_j,\qquad
 \ell_S(a_j)+\ell_S(b_j)\leq r\qquad(1\leq j\leq p).
\end{equation}
Define group elements
\[
 t_j=a_j^{-1}b_j,\qquad s_0=1,\qquad s_j=t_1\cdots t_j.
\]
Then $|t_j|_\Gamma\leq r$. We claim that
\begin{equation}\label{eq:bounded}
 |s_j|_\Gamma\leq D\qquad(0\leq j\leq p).
\end{equation}
Suppose otherwise, and let $h$ be the first index with $|s_h|_\Gamma>D$. Thus
\[
 |s_j|_\Gamma\leq D\ (j<h),\qquad |s_h|_\Gamma\leq D+r.
\]
Choose $z_j\in\mathcal E(x_j,w)\cap\univ$ for $0\leq j\leq h$.

Fix $j\leq h$ and $v\in B_\Gamma(L)$, and put $u=vs_j^{-1}$. Then $|u|_\Gamma\leq L+D+r=w$. For every $1\leq k\leq j$,
\[
 |us_{k-1}a_k^{-1}|_\Gamma
 \leq L+D+r+D+r=L+2D+2r\leq2w.
\]
It follows from \eqref{eq:buffer} that $qus_{k-1}a_k^{-1}\in M$. Apply this monoid element to the $k$th equality in \eqref{eq:chain}. Since $s_k=s_{k-1}a_k^{-1}b_k$, we obtain
\[
 (qus_{k-1})x_{k-1}=(qus_k)x_k.
\]
All coefficients here belong to $M$: they are the products of the applied monoid element with $a_k$ and $b_k$. Induction gives
\begin{equation}\label{eq:transport}
 (qu)x_0=(qus_j)x_j=(qv)x_j.
\end{equation}
For $1\leq m\leq L$, both $u$ and $v$ are prescribed coordinates. Equations~\eqref{eq:shift}, \eqref{eq:E}, and~\eqref{eq:transport} give
\[
 (\sigma_{s_j}z_0)_{v,m}
 =(z_0)_{vs_j^{-1},m}
 =\lambda_m((qu)x_0)
 =\lambda_m((qv)x_j)
 =(z_j)_{v,m}.
\]
Thus $P_L(\sigma_{s_j}z_0)=P_L(z_j)$. Both configurations belong to $\univ$, so Lemma~\ref{lemma:finiteDependence} implies
\[
 z_0,\sigma_{s_1}z_0,\ldots,\sigma_{s_h}z_0\in W_i.
\]
Their successive group displacements have length at most $r$. Proposition~\ref{prop:dad} gives $|s_h|_\Gamma\leq D$, a contradiction. This proves \eqref{eq:bounded}.

Finally, for every $k\leq p$ we have
\[
 |s_{k-1}a_k^{-1}|_\Gamma\leq D+r\leq2w.
\]
Apply $qs_{k-1}a_k^{-1}\in M$ to the $k$th equality in \eqref{eq:chain}. The same calculation, now starting with coefficient $q$, gives
\[
 qx_0=(qs_p)x_p.
\]
By Lemma~\ref{lemma:uniqueSink},
\begin{equation}\label{eq:freeBound}
 d(x_0,x_p)\leq\ell_S(q)+
 \max_{s\in B_\Gamma(D)}\ell_S(qs)=:R_\free.
\end{equation}
Every $qs$ in this formula lies in $M$ by \eqref{eq:buffer}. The finite maximum uses positive monoid word lengths. It depends only on $M,S,K,r$, because $q,D$ were fixed before the action. Any two points of an $r$-component can be the endpoints of such a chain, so this finishes the proof.
\end{proof}

    \subsection*{Acknowledgments}
We thank Jianchao Wu for helpful discussions. 

\printbibliography

\end{document}